\documentclass{amsart}

\usepackage{graphicx}
\usepackage{mathtools}

\newtheorem{theorem}{Theorem}[section]
\newtheorem{lemma}[theorem]{Lemma}

\theoremstyle{definition}
\newtheorem{definition}[theorem]{Definition}

\newtheorem{remark}[theorem]{Remark} 
\theoremstyle{corollary}
\newtheorem{corollary}[theorem]{Corollary}

\theoremstyle{proposition}
\newtheorem{proposition}[theorem]{Proposition}

\theoremstyle{conjecture}
\newtheorem{conjecture}[theorem]{Conjecture}

\theoremstyle{conditionalproposition}

\numberwithin{equation}{section}

\begin{document}

\title[Eigenvalue contour lines of Kac--Murdock--Szeg\H{o} matrices]{ Eigenvalue contour lines of Kac--Murdock--Szeg\H{o} matrices with a complex parameter}

\author{George Fikioris and Christos Papapanos}

\address{School of Electrical and Computer engineering, National Technical University of Athens, GR 157-73 Zografou, Athens, Greece. Christos Papapanos is also with the Department of Electrical Engineering and Computer Sciences, University of California, Berkeley, CA, USA.
}
\email{gfiki@ece.ntua.gr, christos\_papapanos@berkeley.edu}

\date{\today}


\maketitle

\begin{abstract}
A previous paper studied the so-called borderline curves of the Kac--Murdock--Szeg\H{o} matrix $K_{n}(\rho)=\left[\rho^{|j-k|}\right]_{j,k=1}^{n}$, where $\rho\in\mathbb{C}$. These are the level curves (contour lines) in the complex-$\rho$ plane on which $K_n(\rho)$ has a type-1 or type-2 eigenvalue of magnitude $n$, where $n$ is the matrix dimension. Those curves have cusps at all critical points $\rho=\rho_c$ at which multiple (double) eigenvalues occur. The present paper determines corresponding curves pertaining to eigenvalues of magnitude $N\ne n$. We find that these curves no longer present cusps; and that, when $N<n$, the cusps have in a sense transformed into loops. We discuss the meaning of the winding numbers of our curves. Finally, we point out possible extensions to more general matrices.
\end{abstract}

\section{Introduction and Preliminaries}

This is the fourth in a series of papers \cite{Fik2018,FikMavr,Fik2020} dealing with the eigenvalues of the Toeplitz matrix 

\begin{equation}\label{matrixdefinition}
K_{n}(\rho)=\left[\rho^{|j-k|}\right]_{j,k=1}^{n}=\begin{bmatrix}
1 & \rho & \rho^2 & \ldots & \rho^{n-1} \\
\rho & 1 & \rho & \ldots & \rho^{n-2} \\
\rho^2 & \rho & 1 & \ldots & \rho^{n-3} \\
\vdots & \vdots & \vdots & \ddots & \vdots \\
\rho^{n-1} & \rho^{n-2} & \rho^{n-3} & \ldots & 1
\end{bmatrix}
\end{equation}
in which $n=3,4,5,\ldots$ and $\rho \in \mathbb{C}$. In the special case $0<\rho<1$, $K_n(\rho)$ is usually called the Kac--Murdock--Szeg\H{o} matrix. Ref. \cite{Fik2018} discusses its history and gives a number of applications. The matrix $K_n(\rho)$ belongs to a number of frequently-investigated classes of matrices. For any fixed $\rho\in\mathbb{C}$, for example, $K_n(\rho)$ is complex-symmetric; for any fixed $\rho\in\mathbb{C}\setminus\mathbb{R}$ it is non-Hermitian and nonnormal; and for any fixed $\rho\in\mathbb{C}$ with $|\rho|>1$, the corresponding Laurent matrix (doubly infinite Toeplitz matrix) does not have a well-defined and bounded symbol, thus complicating studies of the spectral behavior of $K_n(\rho)$ \cite{Fik2018, Fik2020, Bottcher-increasing}.  Furthermore, the matrix elements are complex-analytic functions of $\rho$, meaning that $K_n(\rho)$ belongs to the class of analytic matrix functions \cite{Fik2020}. As noted in \cite{Fik2018} and \cite{Fik2020}, and as will be further suggested in the present work, 
studies of $K_n(\rho)$ have initiated investigations pertaining to
more general matrices, and have served as means of illustrating the results of such investigations.
This has been true ever since the original work by Kac, Murdock, and Szeg\H{o} \cite{KMS, Grenander} and~ remains~true~today \cite{Fik2020, Bogoya}. 

This work makes little use of \cite{Fik2020} and can be considered a direct continuation and generalization of \cite{FikMavr}. It exploits the facts and associated terminology that we list below.
 
\begin{itemize} 

\item{Each eigenvalue of $K_n(\rho)$ is of one of two types, called \textit{type-1} and \textit{type-2.} The most notable distinguishing feature is that type-1 (type-2) eigenvalues are associated with skew-symmetric (symmetric) eigenvectors. More information on the two types can be found in Theorems 3.7 and 4.1 of \cite{Fik2018}, and Remark 4.2 of \cite{Fik2018}.} 

\item{For $k=1$ or $k=2$, the \textit{borderline curve} $B_n^{(k)}$ is the level curve (contour line) which consists of all $\rho$ in the complex plane for which $K_n(\rho)$ has a type-$k$ eigenvalue whose magnitude is equal to the matrix dimension $n$ \cite{FikMavr}.  The curves $B_n^{(1)}$ and $B_n^{(2)}$ are closed curves \cite{FikMavr} and are conjectured in \cite{FikMavr} to be Jordan curves (with no self-intersections).} 

\item{For certain $n$-dependent values of $\rho$, $K_n(\rho)$ possesses repeated eigenvalues. Apart from certain trivial cases, all repeated eigenvalues are (algebraically) double eigenvalues equal to $-n$, and the critical points $\rho=\rho_c$ for which such \textit{borderline/double} eigenvalues occur satisfy $\rho_c\in \mathbb{C}\setminus \mathbb{R}$ \cite{FikMavr}. The aforementioned trivial cases are described in Section~3 of \cite{FikMavr}. As with any eigenvalue, each borderline/double one is either of type-1 or of type-2. Any type-$k$ critical point $\rho_c$, of course, belongs to the borderline curve $B_n^{(k)}$.}
\end{itemize}

For $n=5$ the two borderline curves are shown in Fig. \ref{fig:Fig_b5k}. The cusp-like curve singularities discerned in the figure (at the points $\rho=\rho_c=\pm i2$ of $B_5^{(1)}$ and $\rho=\rho_c\cong \pm 1.247\pm i1.456$ of $B_5^{(2)}$) are true cusps \cite{Fik2020} that signal the appearance of double eigenvalues. More generally, the curve $B_n^{(k)}$ has a cusp singularity at $\rho=\rho_c$ iff $K_n(\rho)=K_n(\rho_c)$ has a type-$k$ borderline/double  eigenvalue \cite{FikMavr,Fik2020}. As previously mentioned, any such eigenvalue equals $-n$. 

For $n=5$, the matrix $K_n(\rho)=K_5(\rho)$ is special because two of its eigenvalues can be found by means of explicit formulas. The two are specifically given by
\begin{equation}
\label{eq:n-equals-5-evalues}
\lambda^{(1)}=\frac{1}{2}\left[2-\rho^2-\rho^4\pm \rho(\rho^2-1)\sqrt{\rho^2+4}\right].
\end{equation}
Explicit computation of the two corresponding $\lambda^{(1)}$-eigenvectors is also possible, showing that the two $\lambda^{(1)}$ in (\ref{eq:n-equals-5-evalues}) are of type-$1$ (hence the superscript). Therefore, the solid line in Fig. \ref{fig:Fig_b5k} can be determined by means of the simple formula (\ref{eq:n-equals-5-evalues}). To illustrate, (\ref{eq:n-equals-5-evalues}) immediately verifies that both $K_5(i2)$ and $K_5(-i2)$ possess a double eigenvalue equal to $-5$.

There seem to be very few cases for which such elementary and explicit formulas can be found,\footnote{When $n=5$, for example, it does not seem possible to \textit{explicitly} determine the three type-2 eigenvalues (which give rise to $B_5^{(2)}$). All known elementary cases are named in Section~1.1 of~ \cite{Fik2020}.} and the level curves in Fig. \ref{fig:Fig_b5k} were  in fact obtained by more general methods that were developed in \cite{FikMavr}. The purpose of the present paper is to generalize those methods to level curves (contour lines) pertaining to type-$k$ eigenvalues of magnitude~$N$; to give a number of properties of the curves thus obtained; to establish the significance of the curves' winding numbers \cite{Roe}; and to point out possible extensions to more general matrices. We denote the type-$k$ level curves of $K_n(\rho)$ by $L_{n,N}^{(k)}$, so that $L_{n,n}^{(k)}=B_{n}^{(k)}$. 
\begin{figure}
\begin{center}
\includegraphics[scale=0.4]{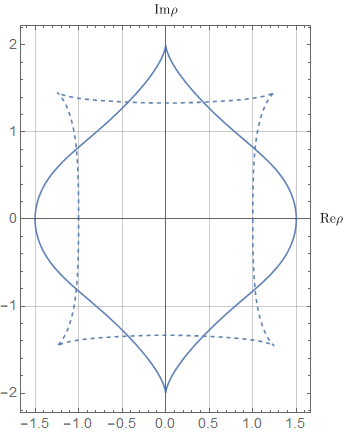}
\end{center}
\caption{Introductory numerical results for $n=5$: Type-1 borderline curve $B_5^{(1)}$ together with type-2 borderline curve $B_5^{(2)}$. The two closed curves are depicted by solid and dashed lines, respectively. The curve singularities at $\rho=\rho_c=\pm i2$ and $\rho=\rho_c\cong \pm 1.247\pm i1.456$ are cusps. When $\rho$ is equal to any such $\rho_c$, $K_n(\rho)=K_5(\rho)=K_5(\rho_c)$ has a borderline/double (type-1 or type-2) eigenvalue equal to $-5$.} \label{fig:Fig_b5k}
\end{figure}

Certain features of $L_{n,N}^{(k)}$ can be predicted beforehand. For any $N$ sufficiently near $n$, eigenvalue continuity makes one expect that $L_{n,N}^{(k)}$ will be close to $B_{n}^{(k)}$.  However,  nontrivial multiple eigenvalues can only have magnitude $n$; for this reason, we can further anticipate that $L_{n,N}^{(k)}$ ($N\ne n$) will no longer exhibit cusps.  

When $N$ is large, an additional prediction can be made. By Section 6.3 of \cite{FikMavr}, for all sufficiently large $|\rho|$ the following hold: (i) There is exactly one type-1 eigenvalue that has magnitude larger than $n$; (ii) there is exactly one type-2 eigenvalue that has magnitude larger than $n$; and (iii) the magnitudes of both these \textit{extraordinary} eigenvalues  are asymptotically approximated by $|\rho|^{n-1}$ as $\rho\to\infty$ \cite{FikMavr} . (Consistently with this expectation, one of the two $\lambda^{(1)}$ in (\ref{eq:n-equals-5-evalues}) satisfies $|\lambda^{(1)}|\sim |\rho|^4$ as $|\rho|\to \infty$.) Accordingly, we can further anticipate that both $L_{n,N}^{(1)}$ and $L_{n,N}^{(2)}$ approach the circle
\begin{equation}
\label{eq:large-N-circle} 
|\rho|=N^\frac{1}{n-1}
\end{equation}
when $N$ is large.

Apart from possible difficulties in numerical implementations, the methods to be developed herein allow $N$ to be arbitrarily large. More interestingly, we allow $N$ to become smaller than $n$. However $N$ must remain larger than a certain threshold to be denoted by $N_\mathrm{min}(n)$. This parameter equals $1$ for $n=3$, approximates $1.1$ for $n=4$, and approaches $0.21n$ for $n=5,6,\ldots$.  As we will see, $N_\mathrm{min}(n)$ is small enough to allow $L_{n,N}^{(k)}$ to differ significantly from $B_{n}^{(k)}$.

The heart of this paper is Theorem \ref{th:main} below, which stems from the following lemma.

\begin{lemma}\label{lemma:lemma3} \cite{FikMavr}
Let $\rho \in \mathbb{C}$. $\lambda \in \mathbb{C}$ is a type-1 (or type-2) eigenvalue of $K_n(\rho)$ iff
\begin{equation}\label{eq:lambda}
\lambda=-\frac{\sin(n\mu)}{\sin{\mu}} \quad \left( {\rm or}\,\,\lambda=\frac{\sin(n\mu)}{\sin{\mu}}  \right),
\end{equation}
where $\mu\in \mathbb{C}$ satisfies
\begin{equation}\label{eq:eigen1eq}
\rho=\frac{\sin{\frac{(n+1)\mu}{2}}}{\sin{\frac{(n-1)\mu}{2}}} \quad \left({\rm or}\,\,  \rho=\frac{\cos{\frac{(n+1)\mu}{2}}}{\cos{\frac{(n-1)\mu}{2}}}\right).
\end{equation}
\end{lemma}
The statement of our Theorem \ref{th:main} resembles that of Theorem 4.1 in \cite{FikMavr}. As already mentioned, there are essential differences in the two theorems' predictions, especially when $n$ approaches $N_\mathrm{min}(n)$. Furthermore,  there are important differences in the theorems' specifics, with the theorem herein being more complicated. Accordingly, much of this paper is devoted to an examination of certain auxiliary functions that appear in the theorem.

\section{Useful properties of Chebyshev polynomials}

We begin with some equalities and inequalities that are most concisely  shown when expressed in terms of $T_k(z)$ and $U_k(z)$ (i.e. the first- and second-kind Chebyshev polynomials of degree $k$):

\begin{lemma}
\begin{equation}
\label{eq:t-def}
T_k(\cos u)=\cos(ku),\quad u\in\mathbb{C},\quad k=0,1,2,\ldots;
\end{equation}
\begin{equation}
\label{eq:u-def}
U_k(\cos u)=\frac{\sin[(k+1)u]}{\sin u},\quad u\in\mathbb{C},\quad k=0,1,2,\ldots;
\end{equation}
\begin{equation}
\label{eq:u-hyperbolic-def}
U_k(\cosh v)=\frac{\sinh[(k+1)v]}{\sinh v},\quad v\in\mathbb{C},\quad k=0,1,2,\ldots;
\end{equation}
\begin{equation}
\label{eq:u-t-identity}
U_{k}(z)=zU_{k-1}(z)+T_{k}(z),\quad z\in\mathbb{C},\quad k=1,2,3\ldots;
\end{equation}
\begin{equation}
\label{eq:t-inequality}
-1\le T_k(x)\le 1,\quad -1\le x\le 1,\quad k=0,1,2,\ldots;
\end{equation}
\begin{equation}
\label{eq:u-t-sum}
U_k(z)=\sum_{m=0}^k z^m\,T_{k-m}(z),\quad z\in\mathbb{C},\quad k=0,1,2,\ldots;
\end{equation}
\begin{equation}
\label{eq:u-inequality}
-(k+1)\le U_{k}(x)\le k+1,\quad -1\le x\le 1,\quad k=0,1,2,\ldots;
\end{equation}
\begin{equation}
\label{eq:t-increases}
1\le x_1<x_2\implies 1 \le T_k(x_1)< T_k(x_2),\quad k=1,2,3,\ldots;
\end{equation}
\begin{equation}
\label{eq:u-increases}
1\le x_1<x_2\implies k+1\le U_{k}(x_1)< U_{k}(x_2),\quad k=1,2,3,\ldots.
\end{equation}
\end{lemma}
\begin{proof}
Eqns. (\ref{eq:t-def}) and (\ref{eq:u-def}) are definitions of the Chebyshev polynomials. Therefore with $z=\cos u$, (\ref{eq:u-t-identity}) reduces to a trigonometric identity. Set $u=iv$ ($v\in\mathbb{C}$) in (\ref{eq:u-def}) to get (\ref{eq:u-hyperbolic-def}). With $\cos u=x$ ($-1\le x\le 1$), (\ref{eq:t-def}) implies (\ref{eq:t-inequality}). 
To verify (\ref{eq:u-t-sum}), use $U_0(z)=T_0(z)=1$, (\ref{eq:u-t-identity}), and induction. Eqn. (\ref{eq:u-inequality}) is trivial when $k=0$; and, for $k\ge 1$, (\ref{eq:u-inequality}) is a consequence of (\ref{eq:u-t-sum}), (\ref{eq:t-inequality}), and the triangle inequality. Setting $u=iv$ ($v\ge 0$) in (\ref{eq:t-def}) shows that, for $k\ge 1$, $T_k(\cosh v)$ is a strictly increasing function of $v$; this fact (together with $T_k(1)=1$) is tantamount to (\ref{eq:t-increases}). Setting $z=x$ ($x\ge 1$) in (\ref{eq:u-t-sum}) and invoking
(\ref{eq:t-increases}), it is seen that $U_k(x)$ is a sum of positive and strictly increasing functions, yielding (\ref{eq:u-increases}). 
\end{proof}

\begin{corollary}
\label{corollary:trig-forms}
The trigonometric/hyperbolic versions of (\ref{eq:u-inequality})/(\ref{eq:u-increases}) are
\begin{equation}
\bigg|\frac{\sin(ku)}{\sin u}\bigg|\le k, \quad u\in\mathbb{R}, \quad k=1,2,3,\ldots,
\end{equation}
\begin{equation}
0\le v_1<v_2\implies k\le \frac{\sinh(kv_1)}{\sinh v_1}<\frac{\sinh(kv_2)}{\sinh v_2},\quad k=2,3,4,\ldots.
\end{equation}
\end{corollary}

We conclude this section with a lemma on the zeros  and maxima of the positive and even function $|U_k(z)|$.

\begin{lemma}
\label{lemma:u-zeros-extrema}
For $k=2,3,4,\ldots$, let $\alpha_{k,m}$ be a zero of $U_k(z)$.  The $\alpha_{k,m}$ are all real, simple, and lie in the interval $(-1,1)$. They are given by
\begin{equation}
\label{eq:u-zeros}
\alpha_{k,m}=\cos\frac{(k+1-m)\pi}{k+1},\quad m=1,2,\ldots,k. 
\end{equation}

Let $\beta_{k,m}$ be a zero of the derivative $U'_k(z)$. The $\beta_{k,m}$ are also real, with
\begin{equation}
\label{eq:u-maxima}
\alpha_{k,m}<\beta_{k,m}<\alpha_{k,m+1}, \quad m=1,2,\ldots,k-1.
\end{equation}
In (\ref{eq:u-zeros}), the indexing is such that  
 $\alpha_{k,m}$ increases when $m$ increases, and ditto for the $\beta_{k,m}$ in (\ref{eq:u-maxima}).

The two numbers $|U_k(\beta_{k,1})|$ and $|U_k(\beta_{k,k-1})|$ are equal and are the largest of the values $|U_k(\beta_{k,m})|$. Furthermore, 
\begin{equation}
\label{eq:u-largest-value}
0<\max_{\alpha_{k,1}\le x\le \alpha_{k,k}}|U_k(x)|=|U_k(\beta_{k,1})|=|U_k(\beta_{k,k-1})|<k+1,
\end{equation}
so that $|U_k(\beta_{k,1})|$ 
is the global maximum of the even function $|U_k(x)|$ in the interval $[\alpha_{k,1},\alpha_{k,k}]$.

When $\alpha_{k,k}\le x\le 1$, the function $U_k(x)$ strictly increases from  $U_k(\alpha_{k,k})=0$ to $U_k(1)=k+1$. And when $-1\le x\le \alpha_{k,1}$, $|U_k(x)|$ strictly decreases from  $|U_k(-1)|=k+1$ to $U_k(\alpha_{k,1})=0$.
\end{lemma}

\begin{proof}
By (\ref{eq:u-def}), the $\alpha_{k,m}$ in (\ref{eq:u-zeros}) are zeros of $U_k(z)$. Since there are $k$ distinct zeros, all are simple. By Rolle's theorem and for $m=1,2,\ldots,k-1$, $U'_k(\beta_{k,m})=0$ for some $\beta_{k,m}$ satisfying (\ref{eq:u-maxima}). Therefore all zeros $\beta_{k,m}$ of $U'_k(z)$ are real and are bracketed by the inequality in (\ref{eq:u-maxima}). By the extreme value theorem and $U_k(\alpha_{k,m})=U_k(\alpha_{k,m+1})=0$, the nonnegative function $|U_k(x)|$ assumes its maximum value within $[\alpha_{k,m},\alpha_{k,m+1}]$ iff $x=\beta_{k,m}$. As $|U_{k-1}(x)|$ is even, we have $|U_k(\beta_{k,1})|=|U_k(\beta_{k,k-1})|$. As discussed in \cite{Mason}, the $|U_k(\beta_{k,m})|$ increase as $|\beta_{k,m}|$ increases away from zero, and all  $|U_k(\beta_{k,m})|$ are smaller than the endpoint values $|U_{k}(\pm 1)|=k+1$, giving (\ref{eq:u-largest-value}).  Since no zero of $U'_k(x)$ lies in the interval $[\alpha_{k,k},1]$, $U_k(x)$ increases strictly in this interval, from its value $0$ at the left endpoint to its value $k+1$ at the right endpoint. The statement pertaining to the interval $[-1,\alpha_{1,k}]$ follows because $|U_k(x)|$ is even.
\end{proof}

\section{The functions $h_{n,N}(v)$ and $g_{n,N}(u)$}

This section defines and gives useful properties of the functions $h_{n,N}(v)$ and $g_{n,N}(u)$, which will play an important role in our main theorem.

\subsection{The function $h_{n,N}(v)$}

\begin{definition} For $n=3,4,\ldots$ and $N\in\mathbb{R}$ with $N>0$, we define 
\begin{equation}
\label{eq:h-def}
h_{n,N}(v)=\sinh^2(nv)-N^2 \sinh^2 {v}, \quad v\ge 0.
\end{equation}
\end{definition}

In the case $N<n$, the behavior that interests us is straightforward. It is given by the following proposition.

\begin{proposition} 
\label{proposition:h-for-Nsmall}
For $N<n$, the function $h_{n,N}(v)$ equals zero when $v=0$, and 
is positive and strictly increasing  when $v\in (0,+\infty)$.
\end{proposition}
\begin{proof}
The equality $h_{n,N}(0)=0$ is obvious. For $v>0$, write
\begin{equation}
\label{eq:h-factorized}
h_{n,N}(v)=[\sinh(nv)+N \sinh{v}][\sinh(nv)-N \sinh{v}].
\end{equation}
The first factor is positive and strictly increasing. The second factor equals 0 when $v=0$, equals $+\infty$ when $v=+\infty$ and has a derivative equal to $n\cosh(nv)-N \cosh{v}$. As $\cosh(nv)>\cosh v>1$ and $n>N$, this derivative cannot vanish. Thus the second factor is also positive and strictly increasing, proving the lemma.
\end{proof}

In the case $N>n$, the behavior of $h_{n,N}(v)$ is not as simple. To describe it, we introduce an auxiliary quantity $v_0(n,N)$. This is defined as the solution to a transcendental equation as follows.

\begin{lemma}
\label{lemma:v0-def}
For $N>n$, the equation
\begin{equation}
\label{eq:temp6}
\frac{\sinh (nv)}{\sinh v}=N, \quad v\ge 0,
\end{equation}
has a unique solution. We denote this solution, which is positive, by $v=v_0(n,N)$. Its limiting value as $N\to n$ is 
\begin{equation}
\lim_{N\to n+0} v_0(n,N)=0.
\end{equation}
\end{lemma}
\begin{proof} The continuous function on the left-hand side of (\ref{eq:temp6}) is strictly increasing in $[0,+\infty)$ by Corollary \ref{corollary:trig-forms}, equals $n$ when $v=0$, and equals $+\infty$ when $v=+\infty$. Since the right-hand side $N$ is greater than $n$ there is a unique solution, and it is positive.
\end{proof}

We can now give our results for $N>n$.

\begin{proposition} 
\label{proposition:h-for-large-N}
For $N>n$, the function $h_{n,N}(v)$ is negative in the interval $(0,v_0(n,N))$, vanishes when $v=0$ or $v=v_0(n,N)$, and is positive and strictly increasing in the interval $(v_0(n,N),+\infty)$.
\end{proposition}

\begin{proof} It is obvious that $h_{n,N}(0)=0$. By (\ref{eq:h-factorized}) and Lemma~\ref{lemma:v0-def}, $h_{n,N}(v)$ vanishes at exactly positive point, namely $v=v_0(n,N)$. The derivative $h'_{n,N}(v)$ is
\begin{equation*}
h'_{n,N}(v)=n\sinh(2nv)-N^2 \sinh{(2v)}, \quad v>0.
\end{equation*}
 By Lemma \ref{lemma:v0-def}, $h'_{n,N}(v)$ also vanishes at exactly one point which, by Rolle's theorem, must lie in the interval $(0,v_0(n,N))$. (This point, in fact, is $v=\frac{1}{2}v_0(n,\frac{N^2}{n})$.) Finally, we can easily show that $h_{n,N}(v)$ is negative as $v\to 0+0$ and positive as $v \to +\infty$. 
These aforementioned properties prove the lemma.
\end{proof}

\subsection{The function $g_{n,N}(u)$}

We now discuss the following function  
$g_{n,N}(u)$.

\begin{definition} For $n=3,4,\ldots$ and $N\in\mathbb{R}$ with $N>0$, we define 
\begin{equation}
\label{eq:g-def}
g_{n,N}(u)=N^2 \sin^2 {u}-\sin^2(nu), \quad -\pi\le u\le \pi.
\end{equation}
\end{definition}

We first assume $N>n$. In this case, what interests us is uncomplicated:
 
\begin{proposition}
\label{proposition:g-for-large-N}
For $N>n$, the function $g_{n,N}(u)$ is zero when  $u=0$ and $|u|=\pi$, and positive when $0<|u|<\pi$. \end{proposition}
\begin{proof}
The statement about zeros is trivial. In the case $0<|u|<\pi$ we have
\begin{equation}
g_{n,N}(u)> n^2\sin^2 u-\sin^2 (nu)\ge 0,
\end{equation}
where we used (\ref{eq:g-def}), the inequalities $N>n$ and $\sin u\ne 0$,  and Corollary \ref{corollary:trig-forms}.
\end{proof}

When $N<n$, the behaviors of interest are best explained in terms of the auxiliary quantities $x'_0(n)$, $N_\mathrm{min}(n)$, $x_0(n,N)$, and $u_0(n,N)$. These are defined via solutions to transcendental equations, as described by the two lemmas that follow.

\begin{lemma}
\label{lemma:x0prime-and-n0-def}
The equation
\begin{equation}
\label{eq:x0prime-def}
U'_{n-1}(x)=0,\quad \cos\frac{2\pi}{n}<x<\cos\frac{\pi}{n}, 
\end{equation}
has a unique solution which we denote by $x'_0(n)$. Further, the $N_\mathrm{min}(n)$ defined by
\begin{equation}
\label{eq:n0-def}
N_\mathrm{min}(n)=\Big|U_{n-1}\left(x'_0(n)\right)\Big|
\end{equation}
satisfies
\begin{equation}
\label{eq:n0-smaller-n}
1\le N_\mathrm{min}(n)<n,
\end{equation}
and is the global maximum attained by $|U_{n-1}(x)|$ in the interval $[-\cos\frac{\pi}{n},\cos\frac{\pi}{n}]$.  
\end{lemma}
\begin{proof}
In Lemma \ref{lemma:u-zeros-extrema}, set $k=n-1$ to obtain $|\alpha_{k,1}|=\alpha_{k,k}=\cos\frac{\pi}{n}$, $\alpha_{k,k-1}=\cos\frac{2\pi}{n}$, $|\beta_{k,1}|=\beta_{k,k-1}=x'_0(n)$, and $N_\mathrm{min}(n)=\big|U_{n-1}\left(x'_0(n)\right)\big|$. Since $\cos\frac{3\pi}{2n}$ belongs to $(\cos\frac{2\pi}{n},\cos\frac{\pi}{n})$ in which $N_\mathrm{min}(n)$ is the maximum value, we have  $N_\mathrm{min}(n)\ge \big|U_{n-1}\left(\cos\frac{3\pi}{2n}\right)\big|$; the first inequality in (\ref{eq:n0-smaller-n}) then follows from (\ref{eq:u-def}).
\end{proof}

\begin{remark}
As we will see, the value $N_\mathrm{min}(n)$ is the smallest level $N$ (of the level curve $L_{n,N}^{(k)}$, for both $k=1$ and $k=2$) for which our main theorem is valid. The previously-mentioned large-$n$ approximation $N_\mathrm{min}(n)\cong 0.21n$ results by substituting $x'_0(n)\cong 3\pi/(2n)$ into (\ref{eq:n0-def}), and using (\ref{eq:u-def}).
\end{remark}

\begin{lemma}
\label{lemma:x0-u0-def}
Let $N$ be such that $N_\mathrm{min}(n)<N<n$, where $N_\mathrm{min}(n)$ is defined in Lemma \ref{lemma:x0prime-and-n0-def}. Then the equation
\begin{equation}
\label{eq:x0-equation}
|U_{n-1}(x)|=N, \quad -1< x < 1,
\end{equation}
has precisely one positive solution, to be denoted by $x_0(n,N)$; and precisely one negative solution equal to $-x_0(n,N)$. 
It is also true that
\begin{equation}
\label{eq:U-n-N-inequality}
|U_{n-1}(x)|
\begin{cases}
<N, \quad |x|<x_0(n,N),\\
=N, \quad |x|=x_0(n,N),\\
>N, \quad x_0(n,N)<|x|\le 1.
\end{cases}
\end{equation}
We further denote
\begin{equation}
\label{eq:u0-def}
u_0(n,N)=\arccos\left[x_0(n,N)\right].
\end{equation}
This quantity satisfies
\begin{equation}
\label{eq:u0-inequality}
0<u_0(n,N)<\frac{\pi}{n}\le\frac{\pi}{3}
\end{equation}
and
\begin{equation}
\label{eq:u0-limiting-value}
\lim_{N\to n-0}u_0(n,N)=0.
\end{equation}
\end{lemma}
\begin{proof}
Since $|U_{n-1}(x)|$ is even, it is sufficient to discuss solutions $x_0(n,N)$ belonging to $[0,1)$, and to show (\ref{eq:U-n-N-inequality}) for $0\le x\le 1$.
Eqn. (\ref{eq:x0-equation}) can have no solution in the subinterval $\big[0,\cos\frac{\pi}{n}\big)$ because $N>N_\mathrm{min}(n)\ge |U_{n-1}(x)|$ for all $x$ in the subinterval, see Lemma \ref{lemma:x0prime-and-n0-def}. There is a unique solution in $\left[\cos\frac{\pi}{n},1\right)$, however, because the left-hand side of (\ref{eq:x0-equation}) increases strictly from $0$ to $n$ (see the last assertion in Lemma~\ref{lemma:u-zeros-extrema}), while the right-hand side $N$ belongs to $(0,n)$. We have thus shown existence, uniqueness, (\ref{eq:U-n-N-inequality}), as well as the inequality 
\begin{equation*}
\cos\frac{\pi}{n}<x_0(n,N)<1,
\end{equation*}
which, via the definition (\ref{eq:u0-def}), gives
(\ref{eq:u0-inequality}).
\end{proof}

We can now describe the desired behavior of $g$ when $N<n$. Besides $N<n$, we also assume $N>N_\mathrm{min}(n)$.

\begin{proposition}
\label{proposition:g-for-Nsmall}
Let $N$ be such that $N_\mathrm{min}(n)<N<n$, where $N_\mathrm{min}(n)$ is defined in Lemma \ref{lemma:x0prime-and-n0-def}. Then the function $g_{n,N}(u)$
satisfies
\begin{equation}
\label{eq:g-positive}
g_{n,N}(u)
\begin{cases}
>0, \quad u_0(n,N)<|u|<\pi-u_0(n,N);\\
=0, \quad u=0,\mathrm{\ or\ } |u|=\pi,\mathrm{\ or\ } |u|=u_0(n,N),\mathrm{\ or\ } |u|=\pi-u_0(n,N);\\
<0,\quad 0<|u|<u_0(n,N),\mathrm{\ or\ }\pi-u_0(n,N)<|u|<\pi.
\end{cases}
\end{equation}
where $u_0(n,N)$ is defined in Lemma \ref{lemma:x0-u0-def}.
\end{proposition}

\begin{proof}
Set $u=\arccos x$ in (\ref{eq:g-def}) and use (\ref{eq:u-def}) to get
\begin{equation}
\label{eq:temp1}
g_{n,N}(\arccos x)=(1-x^2)\Big(N+|U_{n-1}(x)|\Big)\Big(N-|U_{n-1}(x)|\Big),\quad -1\le x\le 1.
\end{equation}
By (\ref{eq:u0-def}), the desired inequality (\ref{eq:g-positive}) amounts to
\begin{equation}
\label{eq:temp2}
g_{n,N}(\arccos x)
\begin{cases}
>0,\quad |x|<x_0(n,N);\\
=0,\quad |x|=1,\mathrm{\ or\ }|x|=x_0(n,N);\\
<0,\quad x_0(n,N)<|x|<1.
\end{cases}
\end{equation}
which holds because the first factor in (\ref{eq:temp1}) is zero when $|x|=1$ and positive otherwise; the second is positive for all $|x|$; and the sign of the third is given~via~(\ref{eq:U-n-N-inequality}).
\end{proof}

\section{Main theorem; some curve properties}

This section presents a theorem that allows one to determine the level curves $L_{n,N}^{(k)}$ of $K_n(\rho)$, and then gives certain rudimentary properties  of the said curves. Given $k$, $n$, and $N$ (and as long as $N>N_\mathrm{min}(n)$), the theorem enables us to compute all complex values $\rho$ that give rise to a type-$k$ eigenvalue with a magnitude equal to $N$.  The desired values of $\rho$ form the range of a complex-valued function $f_{n,N}^{(1)}(u)$ [or $f_{n,N}^{(2)}(u)$], where $u$  belongs to $[-\pi, \pi]$, or to an explicitly defined subinterval of $[-\pi, \pi]$. The functions $f_{n,N}^{(1)}(u)$ and $f_{n,N}^{(2)}(u)$, finally,  are defined via the unique solution to a certain transcendental equation.

\begin{theorem}\label{th:main}
Let $N\in\mathbb{R}$ be such that $N_\mathrm{min}(n)<N<+\infty$, where $N_\mathrm{min}(n)$ is given in Lemma \ref{lemma:x0prime-and-n0-def}. For $\rho \in \mathbb{C}$, $K_{n}(\rho)$ possesses a type-1 eigenvalue $\lambda$ of magnitude $|\lambda|=N$ iff $\rho=f_{n,N}^{(1)}(u)$ and $\lambda=b_{n,N}^{(1)}(u)$ where
\begin{equation}\label{eq:rhosdef}
f_{n,N}^{(1)}(u)=\frac{\displaystyle\sin{\frac{(n+1)\mu(n,N,u)}{2}}}{\displaystyle\sin{\frac{(n-1)\mu(n,N,u)}{2}}},\qquad b_{n,N}^{(1)}(u)=-\frac{\sin[n\mu(n,N,u)]}{\sin\mu(n,N,u)},
\end{equation}
in which 
\begin{equation}\label{eq:mupar}
\mu(n,N,u)=u+i v(n,N,u).
\end{equation}
In \eqref{eq:mupar}, $u$ varies within $[-\pi, \pi]$ as specified below. With $g_{n,N}(u)$ defined in (\ref{eq:g-def}), the function $v(n,N,u)$ is the unique root of the transcendental equation
\begin{equation}\label{eq:ba}
\sinh^2(nv)-N^2 \sinh^2{v}=g_{n,N}(u),
\end{equation}
with the desired root belonging to $(0,+\infty)$ or  $[0,+\infty)$ as also specified below.\\

\noindent \textbf{Case 1:} When $N>n$, the $v(n,N,u)$ is the unique \textbf{positive} root of the transcendental equation (\ref{eq:ba}), for all $u\in[-\pi,\pi]$. Furthermore,
\begin{equation}
\label{eq:v-greater-v0}
v(n,N,u)\ge v_0(n,N),\quad u\in 
[-\pi,\pi],
\end{equation}
where $v_0(n,N)$ is defined in Lemma \ref{lemma:v0-def}. In (\ref{eq:v-greater-v0}), equality occurs iff $u=0$ or $|u|=\pi$.
\\

\noindent \textbf{Case 2:} When $N_\mathrm{min}(n)<N\le n$, the $v(n,N,u)$ is the unique \textbf{nonnegative} root of the transcendental equation (\ref{eq:ba}). Here, the values of $u$ are limited according to 
\begin{equation}
\label{eq:u-limited}
u_0(n,N)\le |u| \le \pi-u_0(n,N),
\end{equation}
in which $u_0(n,N)$ is defined in Lemma \ref{lemma:x0-u0-def} (when $N=n$, use the limiting value (\ref{eq:u0-limiting-value})).\\

Similarly, $K_{n}(\rho)$ possesses a borderline type-2 eigenvalue $\lambda$ iff $\rho=f_{n,N}^{(2)}(u)$ and $\lambda=b_{n,N}^{(2)}(u)$ where
\begin{equation}\label{eq:rhocdef}
f_{n,N}^{(2)}(u)=\frac{\cos{\displaystyle\frac{(n+1)\mu(n,N,u)}{2}}}{\cos{\displaystyle\frac{(n-1)\mu(n,N,u)}{2}}},\qquad b_{n,N}^{(2)}(u)=\frac{\sin[n\mu(n,N,u)]}{\sin\mu(n,N,u)},
\end{equation}
in which $u$, $v(n,N,u)$, and $\mu(n,N,u)$ are exactly the same as in the above-discussed type-1 case. 
\end{theorem}

\begin{proof} 
For $N=n$, the theorem reduces to Theorem 4.1 of \cite{FikMavr} and Lemma 4.2(ii) of \cite{FikMavr}. We thus assume that $N\ne n$ throughout. Accordingly, Case 2 amounts to $N_\mathrm{min}(n)<N<n$.

In Lemma \ref{lemma:lemma3}, take $|\lambda|=N$ and  
set $u={\rm Re}\mu$ and $v={\rm Im}\mu$ to see that $K_{n}(\rho)$ has a borderline type-1 eigenvalue $\lambda$ iff $\rho=\rho_{n,N}(u,v)$ and $\lambda=\lambda_{n,N}(u,v)$ where
\begin{equation}\label{eq:rho1uv}
\rho_{n,N}(u,v)=\frac{\displaystyle\sin{\frac{(n+1)u}{2}}\cosh{\frac{(n+1)v}{2}}+i\cos{\frac{(n+1)u}{2}}\sinh{\frac{(n+1)v}{2}}}
{\displaystyle\displaystyle\sin{\frac{(n-1)u}{2}}\cosh{\frac{(n-1)v}{2}}+i\cos{\frac{(n-1)u}{2}}\sinh{\frac{(n-1)v}{2}}},\quad u,v \in \mathbb{R},
\end{equation}
and
\begin{equation}\label{eq:lambdannn}
\lambda_{n,N}(u,v)=-\frac{\sin[n(u+iv)]}{\sin(u+iv)}
,\quad u,v \in \mathbb{R},
\end{equation}
where $u$, $v$, and $n$ are interrelated via
\begin{equation}\label{eq:uv}
N^2=\frac{\sin^2(nu)+\sinh^2(nv)}{\sin^2{u}+\sinh^2{v}}, \quad u,v \in \mathbb{R}.
\end{equation}
Since the right-hand sides of (\ref{eq:rho1uv})--(\ref{eq:uv}) are $2\pi$-periodic in $u$, we assume $u \in [-\pi, \pi]$ with no loss of generality. Since, also,  $\rho_{n,N}(-u,-v)=\rho_{n,N}(u,v)$ and $\lambda_{n,N}(-u,-v)=\lambda_{n,N}(u,v)$, we further assume $v \in [0,+\infty)$. We must now consider the two cases separately.\\

\noindent \textbf{Case 1:} When $N>n$, 
we limit ourselves to $v \in (0,+\infty)$: For $v=0$, the right-hand side of the transcendental equation (\ref{eq:uv}) is smaller than or equal to $n^2$ by Corollary~\ref{corollary:trig-forms}; and $n^2$ is, in turn, smaller than $N^2$. 
Thus no $u\in [-\pi,\pi]$ can satisfy (\ref{eq:uv})  and we take $v>0$. 

By Proposition \ref{proposition:g-for-large-N}, the function $g_{n,N}(u)$ in the right-hand side of (\ref{eq:ba}) is nonnegative.
By Proposition~\ref{proposition:h-for-large-N}, when $v>0$ the function $h_{n,N}(v)$ in the left-hand side of (\ref{eq:ba}) is nonnegative only when $v\in [v_0(n,N),\infty)$, and is strictly increasing there. 
We have thus shown that (\ref{eq:ba}) has a unique positive solution $v$---which we denote by $v=v(n,N,u)$---and we have also shown (\ref{eq:v-greater-v0}).
In (\ref{eq:v-greater-v0}), equality occurs iff $g(n,N)(u)=0$ which, by Proposition \ref{proposition:g-for-large-N}, is equivalent to $u=0$ or $|u|=\pi$.

For Case~1, (\ref{eq:uv}) is therefore equivalent to the definition \eqref{eq:g-def} and the transcendental equation \eqref{eq:ba}, with $v\in (0,+\infty)$ and $u\in [-\pi,\pi]$. \\

\noindent \textbf{Case 2:} When $N_\mathrm{min}(n)<N<n$ our reasoning is similar.  To begin with, we exclude the values $u=0$ and $|u|=\pi$, for which the right-hand side of (\ref{eq:uv}) is larger than or equal to $n^2$ by Corollary \ref{corollary:trig-forms}.  

By Proposition \ref{proposition:h-for-Nsmall}, the function $h_{n,N}(v)$ on the left-hand side of the transcendental equation (\ref{eq:ba}) is non-negative and strictly increasing. Thus, (\ref{eq:ba}) has a unique solution iff $g_{n,N}(u)\ge 0$ which, by Proposition \ref{proposition:g-for-Nsmall}, is equivalent to $u=0$, $|u|=\pi$, or $u_0(n,N)\le |u|\le \pi-u_0(n,N)$. As the first two cases have been excluded, we have shown (\ref{eq:u-limited}). Thus in Case~2 too, (\ref{eq:uv}) is equivalent to \eqref{eq:g-def} and \eqref{eq:ba}; but here, $v\in[0,\infty)$ and $|u|$ varies in the interval specified in (\ref{eq:u-limited}).\\

With $v=v(n,N,u)$ thus determined, the $\rho_{n,N}(u,v)$  of (\ref{eq:rho1uv}) and the $\lambda_{n,N}(u,v)$ of (\ref{eq:lambdannn}) are no longer functions of $v$, and the notations $f_{n,N}^{(1)}(u)=\rho_{n,N}(u,v)$ and $b_{n,N}^{(1)}(u)=\lambda_{n,N}(u,v)$ prove (\ref{eq:rhosdef}) with (\ref{eq:mupar}). We have thus shown all assertions pertaining to type-1 eigenvalues. 

For the type-2 case, proceed as before with 
\begin{equation*}
\rho_{n,N}(u,v)=\frac{\displaystyle\cos{\frac{(n+1)u}{2}}\cosh{\frac{(n+1)v}{2}}-i\sin{\frac{(n+1)u}{2}}\sinh{\frac{(n+1)v}{2}}}
{\displaystyle\displaystyle\cos{\frac{(n-1)u}{2}}\cosh{\frac{(n-1)v}{2}}-i\sin{\frac{(n-1)u}{2}}\sinh{\frac{(n-1)v}{2}}},\quad u,v \in \mathbb{R},
\end{equation*}
%
\end{proof}

For brevity, denote by $R(n,N)$ the range in which $u$ is supposed to vary in Cases~1 and 2, so that
\begin{equation}
\label{eq:range-of-u}
R(n,N)=\begin{cases}[-\pi,\pi],\quad N>n\\
[-\pi+u_0(n,N),-u_0(n,N)] \cup [u_0(n,N),\pi-u_0(n,N)],\  N_\mathrm{min}(n)<N\le n
\end{cases}
\end{equation}
We this notation, we can re-state the essence of Theorem~\ref{th:main} as follows.
\begin{corollary}
\label{th:curves}
Let $N\in\mathbb{R}$ with $N>N_\mathrm{min}(n)$ and let $k=1$ or $k=2$. 
The matrix $K_{n}(\rho)$ possesses a type-$k$ eigenvalue of magnitude $N$ iff $\rho \in L_{n,N}^{(k)}$, where $L_{n,N}^{(k)}$ is the curve given by
\begin{equation}\label{eq:Ls}
L_{n,N}^{(k)}=\left\{\rho \in \mathbb{C}: \rho=f^{(k)}_{n,N}(u)\mathrm{\ for\ some\ } u\in R(n,N)\right\},
\end{equation}
in which $ f^{(k)}_{n,N}(u)$ is defined in Theorem \ref{th:main}. 
\end{corollary}

The proposition that follows gives some elementary properties of the $L_{n,N}^{(k)}$. 
\begin{proposition}\label{th:symmetries}
The level curves $L_{n,N}^{(1)}$ and $L_{n,N}^{(2)}$ exhibit the following properties.\\
(i) For $k=1$ and $k=2$, $L_{n,N}^{(k)}$ intersects the real axis exactly twice.\\
(ii) Both $L_{n,N}^{(1)}$ and $L_{n,N}^{(2)}$ are symmetric with respect to the real $\rho$-axis.\\
(iii) The union $L_{n,N}^{(1)}\cup L_{n,N}^{(2)}$ is symmetric with respect to the origin $\rho=0$.\\
(iv) For $n=3,5,7\ldots$, both $L_{n,N}^{(1)}$ and $L_{n,N}^{(2)}$ are symmetric with respect to the imaginary $\rho$-axis.\\
(v) For $n=4,6,8\ldots$,  $L_{n,N}^{(1)}$ and $L_{n,N}^{(2)}$ are mirror images of one another with respect to the imaginary $\rho$-axis.
\end{proposition}
\begin{proof}
The proof is very similar to the proof of Proposition 4.4 of \cite{FikMavr}. 
\end{proof}

We end this section with a lemma that we will use shortly. It states that positive (negative) $u$-values correspond to points  $L_{n,N}^{(k)}$ that lie in the lower-half (upper-half) plane. It also gives the $u$-values for which the intersections with the real axis occur. 
\begin{lemma}
\label{lemma:split-plane}
Let $u\in R(n,N)$, let $N>N_\mathrm{min}(n)$, and let $k=1$ or $k=2$. Then
\begin{equation}
u>0\iff \mathrm{Im}f_{n,N}^{(k)}(u)<0
\end{equation}
Furthermore, the $u$-values for which $L_{n,N}^{(k)}$ intersects the real axis are $u=0$ and $|u|=\pi$ in Case~1 ($N>n$), and $|u|= u_0(n,N)$ and $|u|= \pi-u_0(n,N)$ in Case~2 ($N_\mathrm{min}<N<n$).
\end{lemma}
\begin{proof}
See the proof of Lemma 4.2 of \cite{FikMavr}. 
\end{proof}

\section{Numerical results; further curve properties}

We begin this section by stating our results in the form of an algorithm that can generate $L^{(k)}_{n,N}$. 

\subsection{Algorithm}

Given $n$, $N$, and a value $k=1$ or $k=2$, a point $\rho \in L^{(k)}_{n,N}$ can be determined as follows. Determine the required range $R(n,N)$ from (\ref{eq:range-of-u}); pick $u\in R(n,N)$; compute $g_{n,N}(u)$ from (\ref{eq:g-def}); solve the transcendental equation (\ref{eq:ba}) for $v=v(n,N,u)$; set $\mu(n,N,u)=u+iv(n,N,u)$; find $f_{n,N}^{(k)}(u)$ from (\ref{eq:rhosdef}) or (\ref{eq:rhocdef}); and set $\rho=f^{(k)}_{n,N}(u)$. Repeat the above process for many $u\in R(n,N)$ until the continuous curve $L^{(k)}_{n,N}$ is depicted.

\subsection{Initial numerical results; uniqueness}

\begin{figure}
\begin{center}
\includegraphics[scale=0.4]{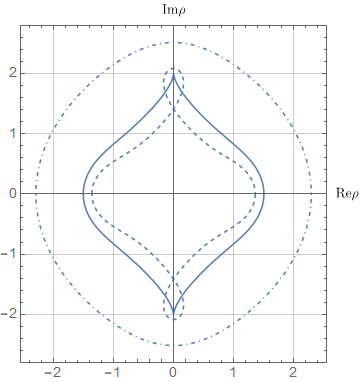}
\end{center}
\caption{Case $n=5$: Type-1 level curves $L^{(1)}_{5,N}$  for $N=5$ (solid line), $N=3$ (dashed line), and $N=30$ (dot-dashed line).} \label{fig:L5-1}
\end{figure}

For $n=5$ and $k=1$, Fig.~\ref{fig:L5-1} shows the $L^{(k)}_{n,N}$ generated by our algorithm for three values of $N$, namely $N=5$, $N=30$, and $N=3$ (note that $3>N_\textrm{min}(5)=1.25$). The $L^{(1)}_{5,5}$ curve is, of course,  the same as the type-1 borderline curve  $B_5^{(1)}$ of Fig. \ref{fig:Fig_b5k}. As anticipated (see  Introduction), singularities appear only in the $N=5$ curve; and the $N=30$ curve is quite close to the circle $|\rho|=30^{1/4}\cong 2.3$.  While the $N=3$ curve presents no singularities, it self-intersects twice, and loops around the two cusps of the $N=5$ curve. Clearly, a self-intersection point of $L^{(k)}_{n,N}$ gives rise to \textit{two} type-$k$ eigenvalues of magnitude $N$. In a sense, the cusp of $B_n^{(k)}=L_{n,n}^{(k)}$ has transformed into the loop of $L_{n,N}^{(k)}$. Conversely, it is illuminating to think of the cusp as a degenerate loop, where the aforementioned two eigenvalues coalesce and give the borderline/double eigenvalue.

Recall that the predictions of Fig.~\ref{fig:L5-1} can be checked via the elementary formulas (\ref{eq:n-equals-5-evalues}). When $\rho=i\sqrt{2}$, for example, (\ref{eq:n-equals-5-evalues}) gives $\lambda^{(1)}=\pm 3i$, meaning that the matrix $K_5(i\sqrt{2})$ has two type-1 eigenvalues of magnitude $3$. This fact is in agreement with Fig.~\ref{fig:L5-1}, in which $\rho=i\sqrt{2}$ is seen to be a self-intersection point of the $N=3$ curve.  As another example, the $N=5$ and $N=3$ curves are seen in the figure to intersect once in the first quadrant. A close focus tells us that the intersection point is  $\rho\cong 0.139 +i1.693$, meaning that the two $\lambda^{(1)}$ must have magnitudes $5$ and $3$ when $\rho$ assumes the aforementioned value. Eqn. (\ref{eq:n-equals-5-evalues}) verifies this is indeed the case (the respective phases are $1.87$ and $-2.12$ rads). 

As illustrated via Fig.~\ref{fig:L5-1}, a self-intersection point $\rho\in L_{n,N}^{(k)}$ gives rise to \textit{two} type-$k$ eigenvalues of magnitude $N$.
When there is a cusp at $\rho\in L_{n,N}^{(k)}$ (this can only occur when $N=n$), we have a \textit{double} type-$k$ eigenvalue (whose magnitude is $n$). Otherwise the situation is simpler: 

\begin{theorem}\label{th:uniqueness}
Let $N\in\mathbb{R}$ with $N>N_\mathrm{min}(n)$, let $k=1$ or $k=2$, and let  $\rho$ be a point of  $L_{n,N}^{(k)}$ that is neither a self-intersection point nor a cusp point of $L_{n,N}^{(k)}$. Then 
$K_{n}(\rho)$ possesses a unique type-$k$ eigenvalue of magnitude $N$, and the said eigenvalue is non-repeated.
\end{theorem}
\begin{proof}
Existence of a type-$k$ eigenvalue $\lambda$ with $|\lambda|=N$ follows from
Corollary~\ref{th:curves}. Suppose that $\lambda^{\prime}$ is also a type-$k$ eigenvalue with $|\lambda^{\prime}|=|\lambda|=N$. By Theorem~\ref{th:main} and Corollary \ref{th:curves}, there exist $u$ and $u^{\prime}$ in $R(n,N)$ such that
\begin{equation*}
\rho=f^{(k)}_{n,N} (u),\quad \lambda=b^{(k)}_{n,N}(u),\quad \rho=f^{(k)}_{n,N} (u^{\prime}), \quad \lambda^{\prime}=b^{(k)}_{n,N}(u^{\prime}).
\end{equation*}
The equality $f^{(k)}_{n,N} (u)=f^{(k)}_{n,N} (u^{\prime})$ implies $u=u^{\prime}$ (otherwise, $\rho$ would be a self-intersection point of $L_{n,N}^{(k)}$). Thus $\lambda=\lambda^{\prime}$ and we have demonstrated uniqueness.

We now show that $\lambda$ is non-repeated. Suppose first that $\rho\in\mathbb{C}\setminus\{-1,0,1\}$. If $\lambda$ were repeated, then $\lambda$ would be a double eigenvalue with $\lambda=-n$ (see Theorem 2.4 of \cite{Fik2020}). Therefore, $\lambda$ would be a borderline/double eigenvalue (see Definition 1.2 of \cite{FikMavr}). This would in turn imply that $\rho$ is a cusp point of $L_{n,N}^{(k)}=L_{n,n}^{(k)}=B_n^{(k)}$ (see our Introduction), contradicting the hypothesis of our theorem.

By $|\lambda|=N>N_\mathrm{min}(n)$ and (\ref{eq:n0-smaller-n}), we have $|\lambda|>1$. By (6.22) of \cite{Fik2018} and Lemma~2.5 of \cite{Fik2018}, the only eigenvalue of $K_n(\pm 1)$ that satisfies $|\lambda|>1$ is simple (in fact, its magnitude equals $n$). Thus $\lambda$ is non-repeated when $\rho=\pm 1$. This completes our proof, for the value $\rho=0$ is not permissible: If $\rho=0$, then $K_n(\rho)=K_n(0)$ would be the $n\times n$ identity matrix, contradicting  $|\lambda|>1$.
\end{proof}

For $n=12$ and $k=2$, Fig. \ref{fig:fig_highN} shows the $L^{(k)}_{n,N}$ for $N=12$, $20$, $30$, $70$, and $200$. An increase in $N$ is seen to result in a slightly larger curve. Observe how the cusp singularities of the smallest ($N=12$) curve\footnote{The number of cusps is $10$, corresponding to $10$ different borderline/double type-2 eigenvalues. The number $10$ can be found \textit{a priori} via Theorem 4.5 of \cite{Fik2018}. Note that no cusp lies on the real axis. This is always so because, for $\rho\in\mathbb{R}$ with $|\rho|>1$, $K_n(\rho)$ has no multiple eigenvalues, see Proposition 6.1 of \cite{Fik2018}.} gradually recede as $N$ grows, with the largest ($N=200$) curve closely approaching the circle $|\rho|=200^{1/11}\cong 1.62$, in accordance with (\ref{eq:large-N-circle}).

\begin{figure}
\begin{center}
\includegraphics[scale=0.4]{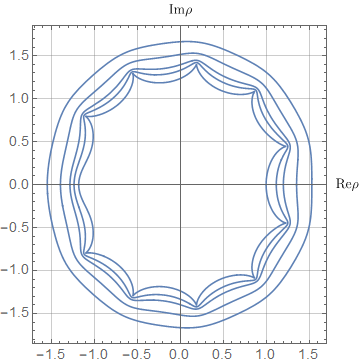}
\end{center}
\caption{Case $n=12$: Type-2 level curves $L^{(2)}_{12,N}$  for $N=12$, $N=20$, $N=30$, $N=70$, and $N=200$. An increase in $N$ enlarges the curve.} \label{fig:fig_highN}
\end{figure}

\subsection{On the existence of loops}

As long as $N_\mathrm{min}(n)<N<n$, loops similar to the ones in Fig.~\ref{fig:L5-1} were observed in all the $L_{n,N}^{(k)}$ generated by our algorithm. It is thus worth proving that such loops always occur. A completely general proof seems difficult, so we limit ourselves to the special case used in Section~6 of \cite{Fik2020} as a means of studying eigenvalue bifurcations. Our special case assumes that $n=3,5,7,\ldots$ and that
\begin{equation}
\label{eq:symmetric-case-assumption}
k=\begin{cases}1,\mathrm{\ if\ }n=5,9,13,\ldots\\
2,\mathrm{\ if\ }n=3,7,11,\ldots
\end{cases}
\end{equation}
Subject to the above assumptions, \cite{Fik2020} shows that $B_n^{(k)}=L_{n,n}^{(k)}$ has  a cusp on the positive-imaginary semi-axis, and one on the negative-imaginary semi-axis, as in the example in Fig. \ref{fig:L5-1}. Theorem \ref{th:loops} will show that the situation is different when $N<n$.

\begin{lemma}
\label{lemma:v-im-def}
Let $N\in\mathbb{R}$ with $1\le N<n$, where $n=3,5,7,\ldots$. Then the equation
\begin{equation}
\cosh(nv)=N\cosh v,\quad v\ge 0
\end{equation}
has a unique nonnegative solution, to be denoted by $v_\mathrm{im}(n,N)$.
\end{lemma}
\begin{proof}
The function $\cosh(nv)-N\cosh v$ equals $1-N<0$ when $v=0$, equals $+\infty$ when $v=+\infty$, and has the derivative $n\sinh(nv)-N\sinh v$, which is positive by Corollary \ref{corollary:trig-forms}.
\end{proof}

The lemma that follows gives four points for which $L_{n,N}^{(k)}$ intersects the imaginary axis. Two of the four are self-intersection points.

\begin{lemma}
\label{lemma:loop}
Let $N_\mathrm{min}(n)<N<n$ with $n=3,5,7,\ldots$, let $k$ be given by (\ref{eq:symmetric-case-assumption}),  and let $u'$ be any one of the six values 
\begin{equation}
\label{eq:six-u-values}
\pm \frac{\pi}{2},\ \pm \frac{\pi}{2}\pm u_0(n,N)
\end{equation}
where $u_0(n,N)$ is given in Lemma \ref{lemma:x0-u0-def}. Then for all six $u'$, $f_{n,N}^{(k)}(u')$ is purely imaginary. Furthermore,
\begin{equation}
\label{eq:temp4}
\rho_\mathrm{upper}\equiv f_{n,N}^{(k)}\left(-\frac{\pi}{2}-u_0(n,N)
\right)=f_{n,N}^{(k)}\left(-\frac{\pi}{2}+u_0(n,N)
\right)
\end{equation}
and
\begin{equation}
\label{eq:temp5}
\rho_\mathrm{lower}\equiv f_{n,N}^{(k)}\left(\frac{\pi}{2}-u_0(n,N)
\right)=f_{n,N}^{(k)}\left(\frac{\pi}{2}+u_0(n,N)
\right)
\end{equation}
\end{lemma}
\begin{proof} 
With the function $v(n,N,u)$ defined in Theorem \ref{th:main}, we first show that
\begin{equation}
\label{eq:temp3}
v(n,N,u')=v_\mathrm{im}(n,N)
\end{equation}
for all six values $u'$, where
$v_\mathrm{im}(n,N)$ is defined in Lemma~\ref{lemma:v-im-def}.
By Theorem \ref{th:main} and eqn. (\ref{eq:h-def}), it suffices to show that
\begin{equation}
\label{eq:temp7}
h_{n,N}\left(v_\mathrm{in}(n,N)\right)=g_{n,N}(u')
\end{equation}
By Lemma~\ref{lemma:v-im-def} and (\ref{eq:h-def}), the left-hand side equals $N^2-1$. By (\ref{eq:six-u-values}), (\ref{eq:g-def}), and (\ref{eq:u-def}), the right-hand side also equals $N^2-1$, thus proving (\ref{eq:temp7}) and (\ref{eq:temp3}). With (\ref{eq:temp3}), (\ref{eq:mupar}) gives
\begin{equation}
\label{eq:temp8}
\mu(n,N,u')=u^{\prime}+iv_\mathrm{im}(n,N)
\end{equation}
for all six values $u^{\prime}$. We prove what remains upon substituting the six $u^{\prime}$ into the expressions for $f_{n,N}^{(k)}(u)$ in (\ref{eq:rhosdef}) and (\ref{eq:rhocdef}), and simplifying the resulting expressions using (\ref{eq:temp8}) and Lemma \ref{lemma:v-im-def}.
\end{proof}

\begin{theorem}
\label{th:loops}
Let $N_\mathrm{min}(n)<N<n$ with $n=3,5,7,\ldots$ and let $k$ be given by (\ref{eq:symmetric-case-assumption}). Then
$L_{n,N}^{(k)}$ loops in both the upper- and the lower-half planes. In the upper- (lower-) half plane, the  self-intersection point is given by the $\rho_\mathrm{upper}$ ($\rho_\mathrm{lower}$) in (\ref{eq:temp4}) or (\ref{eq:temp5}). Both self-intersection points lie on the imaginary axis. In the limit $N\to n-0$, furthermore, each loop shrinks to a single point, the point being the corresponding cusp of $B_n^{(k)}=L_{n,n}^{(k)}$ on the imaginary axis.
\end{theorem}
\begin{proof} The point set
\begin{equation}\label{eq:loop-upper}
\left\{\rho \in \mathbb{C}: \rho=f^{(k)}_{n,N}(u)\mathrm{\ for\ some\ } u\in \left[-\frac{\pi}{2}-u_0(n,N),-\frac{\pi}{2}+u_0(n,N)\right]\right\},
\end{equation}
is a proper subset of $L_{n,N}^{(k)}$ by Corollary \ref{th:curves} and (\ref{eq:u0-inequality}), and belongs to the upper-half plane by Lemma~\ref{lemma:split-plane}. The endpoints of the interval in (\ref{eq:loop-upper}) are different by (\ref{eq:u0-inequality}). By (\ref{eq:temp4}), however, the two endpoints correspond to the same point $\rho_\mathrm{upper}$ (on the positive imaginary semi-axis). Therefore (\ref{eq:loop-upper}) is the parametric representation of a curve that is closed, and the point $\rho$ is a self-intersection point of $L_{n,N}^{(k)}$. In other words, our point set forms a loop of  $L_{n,N}^{(k)}$. By (\ref{eq:u0-limiting-value}), the loop shrinks to a point (on the imaginary axis) in the limit $N\to n-0$, completing our proof for the upper-half-plane loop. For a proof corresponding to the lower-half plane, replace the two instances of $-\frac{\pi}{2}$ in (\ref{eq:loop-upper}) by $\frac{\pi}{2}$ .
\end{proof}
\begin{figure}
\begin{center}
\includegraphics[scale=0.4]{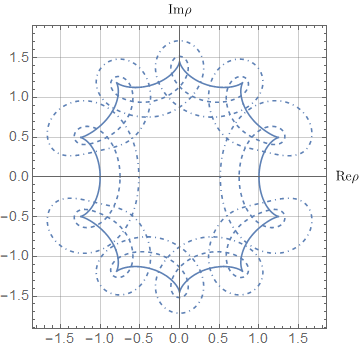}
\end{center}
\caption{Type-2 level curves $L^{(2)}_{11,N}$  for $N=11$ (solid line), $N=5$ (dashed line), and $N=2.8$ (dot-dashed line).} \label{fig_L11-2}
\end{figure}

For $n=11$, Fig. \ref{fig_L11-2} gives the three type-2 curves with $N=11$, $N=5$, and $N=2.8$. The borderline curve $B_{11}^{(2)}=L_{11,11}^{(2)}$ exhibits a number of cusps, two of which lie on the imaginary axis. The two other curves loop around each of those cusps. As expected, the $N=5$ curve exhibits smaller loops than does the $N=2.8$ curve; and the former curve is closer to the borderline ($N=n=11$) curve.

As we always observed loops when $N_\mathrm{min}(n)<N<n$, proved the existence of loops in a special case (Theorem \ref{th:loops}), and never observed loops when $N\ge n$, we are led to the conjecture that follows, which is a generalization (to cases $N\ne n$) of Conjecture 4.5 of \cite{FikMavr}.

\begin{conjecture}\label{th:jordan} The closed curves
$L_{n,N}^{(1)}$ and $L_{n,N}^{(2)}$ are Jordan curves whenever $N\ge n$, but non-Jordan whenever $N_\mathrm{min}(n)<N<n$.
\end{conjecture}

\section{Eigenvalues with magnitudes larger than $N$; winding numbers}
\label{sec: number-j}

For $N>N_\mathrm{min}(n)$, and for a given $\rho\in\mathbb{C}$, this section illustrates how a given curve $L_{n,N}^{(k)}$  can help one find the non-negative integer $j_{n,N}^{(k)}(\rho)$ defined as follows.

\begin{definition} Let $\rho\in\mathbb{C}$ and $N>0$, and let $k=1$ or $k=2$. By $j_{n,N}^{(k)}(\rho)$ we denote the number (counting multiplicities) of type-$k$ eigenvalues of $K_n(\rho)$ whose magnitudes are larger than $N$.   
\end{definition}

Section~6 of \cite{FikMavr} treats the special case $N=n$ and Conjecture~\ref{th:jordan} allows us to extend that treatment to cases $N>n$. When $N_\mathrm{min}(n)<N<n$, however, the self-intersections of $L_{n,N}^{(k)}$ render the determination of $j_{n,N}^{(k)}(\rho)$ more involved. Although we use some of the principles discussed in Section~6 of \cite{FikMavr}, what follows additionally involves the winding numbers \cite{Roe} associated with our curve. Accordingly, we provide $L_{n,N}^{(k)}$ with an orientation; this allows us to find winding numbers, and to distinguish between points that are just to the left of the curve and points that are just to the right. 
The orientation we choose is such that the curve crosses the positive real axis from the lower- to the upper-half plane; this initial orientation gives the one at any other curve point in a natural manner. As we will see, the opposite orientation will do just as well.

\begin{figure}
\begin{center}
\includegraphics[scale=0.4]{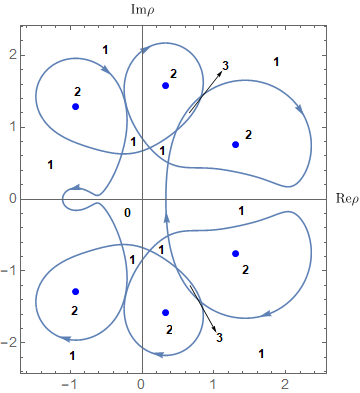}
\end{center}
\caption{Type-2 oriented level curve $L^{(2)}_{8,1.85}$ (solid line with arrows), together with points showing the cusp locations of $B^{(2)}_8$. The value $N=1.85$ is close to $N_\mathrm{min}(8)\cong 1.833$. (At first sight, it might appear that there are self-intersection points in the second and third quadrants, but a close focus shows this to be untrue.) The curve splits the plane into components, and the numbers ($0$, $1$, $2$, or $3$) shown in the figure are the $j_{8,1.85}^{(2)}(\rho)$ of each component.} \label{fig_5}
\end{figure}

Fig. \ref{fig_5} shows $L_{n,N}^{(k)}$ thus oriented for $k=2$, $n=8$, and $N=1.85$. The six points pictured in Fig.~\ref{fig_5} correspond to the six cusps of the borderline curve $B_n^{(2)}=L_{n,n}^{(2)}$. (Recall that each cusp point gives rise to a borderline/double eigenvalue equal to $-n$.) The value $N=1.85$ is slightly larger than $N_\mathrm{min}(8)\cong 1.833$; as a result, there are significant deviations from $B_n^{(2)}$.   

The curve $L_{n,N}^{(k)}$ splits the complex-$\rho$~plane into a number of disjoint components. We have labeled each component by its $j_{n,N}^{(k)}(\rho)$ (which is constant within the component, see Section 6.1 of \cite{FikMavr}).  It is seen that  $j_{n,N}^{(k)}(\rho)=0,1,2,$ or $3$, with the two components labeled $3$ being very small. We see, also, that any component containing a cusp point has $j_{n,N}^{(k)}(\rho)=2$; thus, at any cusp point $\rho_c$, all type-$2$ eigenvalues of $K_8(\rho_c)$ have magnitudes smaller than $N=1.85$, with the exception of the borderline/double eigenvalue (which has magnitude $n=8$ and is counted twice). 

The reader is invited to show that the $j_{n,N}^{(k)}(\rho)$ in Fig. \ref{fig_5} can be determined using the two rules that follow.\\

\noindent \underline{Rule 1:} The unique component that extends to infinity has $j_{n,N}^{(k)}(\rho)=j_{n,N}^{(k)}(\infty)=1$.\\

\noindent \underline{Rule 2:} Crossing 
$L_{n,N}^{(k)}$ via a non-self-intersection point of $L_{n,N}^{(k)}$ results in changing $j_{n,N}^{(k)}(\rho)$ by $1$, with the larger  $j_{n,N}^{(k)}(\rho)$ appearing: (i) on the right side of the curve for our choice of orientation; but (ii) on the left side of the curve for the opposite orientation.\\

Let us justify these rules. As $L_{n,N}^{(k)}$ is bounded, only one component extends to infinity. Rule 1 is true because $N>N_\mathrm{min}(n)>1$ by (\ref{eq:n0-smaller-n}); and because, as $|\rho|\to\infty$, exactly one type-$k$ eigenvalue is unbounded, while all other type-$k$ eigenvalues approach $-1$, see (6.6) of \cite{FikMavr}.

The change by $1$ specified by Rule~2 is evident via Theorem \ref{th:uniqueness}; see also Section 6.2 of \cite{FikMavr}. If we start in the unbounded component and cross the curve via the positive real axis, we enter the component containing the origin. This component has $j_{n,N}^{(k)}(\rho)=j_{n,N}^{(k)}(0)=0$ because $K_n(0)$ is the $n\times n$ identity matrix, whose eigenvalues are all smaller than $N$ by (\ref{eq:n0-smaller-n}) and $N>N_\mathrm{min}(n)$. For the provided orientation, therefore, crossing from right to left is accompanied by a \textit{decrease} in  $j_{n,N}^{(k)}(\rho)$, as stated in Rule~2. The reverse would be true if we had chosen the opposite orientation.

Evidently, our two rules are similar 
in nature to those used in computing winding numbers of closed, bounded, and oriented curves \cite{Roe}. The latter rules, which appear in the literature in various contexts \cite{Erickson, Carter}, are often referred to as Alexander numbering, the name originating from Alexander's 1928 paper \cite{Alexander}. More precisely, for the orientation provided in Fig. \ref{fig_5}, let $\mathrm{wind}\left(L_{n,N}^{(k)}, \rho\right)$ be the winding number of $L_{n,N}^{(k)}$ with respect to the point $\rho\notin L_{n,N}^{(k)}$.  At infinity, in place of Rule~1 we have $\mathrm{wind}\left(L_{n,N}^{(k)}, \infty\right)=0$. Furthermore, $\mathrm{wind}\left(L_{n,N}^{(k)}, \rho\right)$ obeys Rule~2 but with the larger of the two $\mathrm{wind}\left(L_{n,N}^{(k)}, \rho\right)$ appearing on the \textit{left} side of the curve  \cite{Roe}--\cite{Carter}. Consequently, $-\mathrm{wind}\left(L_{n,N}^{(k)}, \rho\right)$ obeys Rule~2. We have thus arrived at the simple relation  
\begin{equation}
\label{eq:winding-1}
j_{n,N}^{(k)}(\rho)=1-\mathrm{wind}\left(L_{n,N}^{(k)}, \rho\right).
\end{equation}

Numerical experiments (specifically, numerical computations of the eigenvalues) verified that our two rules---or their equivalent
 (\ref{eq:winding-1})---correctly gave the $j_{n,N}^{(k)}(\rho)$ for all the $L_{n,N}^{(k)}$ we generated via our algorithm. In particular, Rule~2 remains true when the non-self-intersection point is a cusp (recall that cusps appear only when $N=n$): as discussed in \cite{Fik2020} and Section~6 of \cite{FikMavr}, any cusp is associated with eigenvalue bifurcations, so that $j_{n,N}^{(k)}(\rho)$ changes by $1$ (rather than $2$, even if there is a double eigenvalue at the cusp). It is also easy to understand why $j_{n,N}^{(k)}(\rho)$ changes by $0$ or $2$ along a trajectory that passes through a self-intersection point, as one can see from Fig. \ref{fig_5}.

Needless to say, any Jordan curve (see Conjecture \ref{th:jordan}) separates the complex plane into two components, namely the interior and the exterior of $L_{n,N}^{(k)}$, with $j_{n,N}^{(k)}(\rho)=j_{n,N}^{(k)}(0)=0$ and $j_{n,N}^{(k)}(\rho)=j_{n,N}^{(k)}(\infty)=1$, respectively, with the value $1$ corresponding to the extraordinary eigenvalue mentioned in our Introduction.

\section{Extensions}

\subsection{An illustrative example}
\label{sec:example}
Via an application arising in physics \cite{Tanaka}, we now show that many of the behaviors exhibited by $K_n(\rho)$ also occur elsewhere. Our example is described by the cubic equation \cite{Tanaka}
\begin{equation}
\label{eq:cubic-eq}
\lambda(\lambda-\rho-\alpha^2)^2+\frac{\pi^2\alpha^4}{4}=0,\quad \alpha>0
\end{equation}
whose unknown is $\lambda$. Consistent with our notation for the Kac--Murdock--Szeg\H{o} matrix $K_n(\rho)$, we consider $\rho\in\mathbb{C}$ to be the varying parameter. As we will shortly discuss in more detail, our $\rho$ corresponds to the quantity $\epsilon_a\in\mathbb{R}$ of \cite{Tanaka}. 
The notation $\alpha$ (as well as the value $0.1$ to which we will soon fix $\alpha$) comes from \cite{Tanaka}.

Conditions for multiple zeros follow from the familiar procedure \cite{DLMF} of setting the discriminant of (\ref{eq:cubic-eq}) equal to zero and solving for $\rho$. We thus find that multiple zeros occur when $\rho$ is equal to one of the three critical points $\rho_c(m)$, where
\begin{equation}
\label{eq:cusp-locations}
\rho_c(m)=-\alpha^2-3\left(\frac{\pi\alpha^2}{4}\right)^\frac{2}{3}\exp\left({i\frac{2m\pi}{3}}\right),\quad m=-1,0,1,
\end{equation}
and that the corresponding zeros are double zeros equal to
\begin{equation}
\lambda_c(m)=-\left(\frac{\pi\alpha^2}{4}\right)^\frac{2}{3}\exp\left({i\frac{2m\pi}{3}}\right),\quad m=-1,0,1
\end{equation}
(see also \cite{Garmon}). It follows that the magnitudes of all three double zeros are equal to $N_0$, with
\begin{equation}
N_0=\left|\lambda_c(m)\right|=\left(\frac{\pi\alpha^2}{4}\right)^\frac{2}{3},\quad m=-1,0,1.
\end{equation}
\begin{figure}
\begin{center}
\includegraphics[scale=0.4]{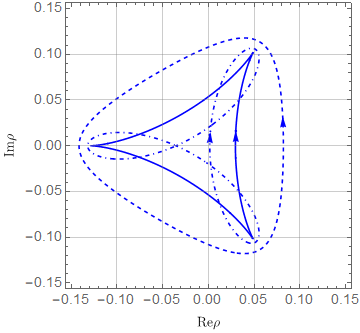}
\end{center}
\caption{Level curves $L_N$ for the three zeros of the cubic equation (\ref{eq:cubic-eq}) with $\alpha=0.1$. The levels of the three curves are $N=N_0$ (solid line), $N=0.5N_0$ (dashed line), and $N=1.4N_0$ (dot-dashed line). The arrows indicate curve orientations.} \label{fig_6}
\end{figure}

Let $\lambda$ denote any of the three solutions of (\ref{eq:cubic-eq}). The $N$-level curves (contour lines on which $|\lambda|=N$) can be numerically generated via the well-known closed-form expressions \cite{DLMF} for the zeros of cubic equations. For $\alpha=0.1$, Fig. \ref{fig_6} shows contour lines $L_{N}$ thus obtained for $N=N_0$ (solid line), $N=0.5N_0$ (dashed line), and $N=1.4N_0$ (dot-dashed line). 
The closed curves of Fig. \ref{fig_6}, which can be compared to those in Fig. \ref{fig:L5-1}, have the following features.\\

\noindent (i) The curve of level $N=N_0$ (but not the other two curves) exhibits three cusp-like singularities.  They occur at the double-eigenvalue positions, namely at the $\rho=\rho_c(m)$ of (\ref{eq:cusp-locations}). Therefore the borderline curve $B_n^{(k)}=L_{n,n}^{(k)}$ of $K_n(\rho)$ is analogous to the $N=N_0$ curve of Fig. \ref{fig_6}, with $n$ corresponding (in this respect) to the double-eigenvalue level $N_0$.\\

\noindent (ii) The curve of level $N=0.5N_0$ envelops the one of level $N=N_0$. Here, as opposed to Fig. \ref{fig:L5-1}, the $N$-value of the surrounding curve is \textit{smaller} than the $N$-value of the enclosed curve.\\

\noindent (iii) The non-Jordan curve (of level  $N=1.4N_0$) loops around the three cusps, similarly to the third curve (of level $N=3$) in Fig. \ref{fig:L5-1}.\\

\noindent (iv) As in Fig. \ref{fig:L5-1}, each curve in Fig. \ref{fig_6} separates the plane into components. Let $j_N(\rho)\in\{0,1,2,3\}$ be the number of zeros of (\ref{eq:cubic-eq}) whose magnitudes exceed $N$. As $|\rho|\to\infty$, the three zeros of (\ref{eq:cubic-eq}) approach $\rho$, $\rho$, and  $0$. Thus in place of Rule~1 of Section~\ref{sec: number-j} we have\\

\noindent \underline{Rule 1$^\prime$:} For each $N$ in Fig. \ref{fig_6}, the unbounded component has $j_N(\rho)=2$.\\

\noindent Let us provide the three curves with the orientations shown in Fig. \ref{fig_6}.
When $\rho=-\alpha^2$, all three zeros of (\ref{eq:cubic-eq}) have magnitudes $2^{2/3}N_0\cong 1.59 N_0$. Therefore for the three values of $N$ used in Fig. \ref{fig_6} we have $j_N(-0.01)=3$. Since the component containing $\rho=-0.01$ is always adjacent to the unbounded one, we can replace Rule~2 of Section~\ref{sec: number-j} by\\

\noindent \underline{Rule 2$^\prime$:} For the three curves of Fig. \ref{fig_6} (oriented as shown), crossing the curve from right to left results in increasing $j_N(\rho)$ by $1$.\\

\noindent (v) For the three oriented curves in Fig. \ref{fig_6}, the equation corresponding to (\ref{eq:winding-1}) is  
\begin{equation}
\label{eq:winding-2}
j_N(\rho)=\mathrm{wind}(L_N,\rho)+2
\end{equation}
where $\mathrm{wind}(L_N,\rho)$ is the winding number of $L_N$ (oriented as in the figure) about any point $\rho\notin L_N$. 

When applied to the three $L_N$ of Fig. \ref{fig_6}, eqn.  (\ref{eq:winding-2}) (which is tantamount to Rules 1$^\prime$ and 2$^\prime$) easily gives any desired $j_N(\rho)$. For example, it is apparent that any component of $L_{1.4N_0}$ that loops around a cusp-like singularity has $j_N(\rho)=1$.

\subsection{Physics of our example} Eqn. (\ref{eq:cubic-eq}) pertains \cite{Tanaka} to an open quantum system consisting of a discrete quantum state of energy $\epsilon_a$ (the symbol $\epsilon_a$ of \cite{Tanaka} corresponds to our $\rho$), coupled with a one-dimensional continuum state. Ref. \cite{Tanaka}, which assumes $\epsilon_a\in\mathbb{R}$, specifically studies the Time-Symmetry Breaking Phase Transition (TSBPT) caused by the non-linearity associated with a Van Hove singularity. Here, our generalization to complex $\epsilon_a$ amounts to introducing a loss (or gain, depending on the sign of $\mathrm{Im}\{\epsilon_a\}$) to the discrete quantum state, similarly to a lossy optical cavity.  As discussed in \cite{Garmon}, the cusp on the real axis corresponds to the bound state, whereas the other two cusps correspond to resonance and anti-resonance.

Contour lines such as the three in Fig. \ref{fig_6} consist of all complex-energy values for which an eigenvalue maintains a constant magnitude $N$. TSBPT occurs along any trajectory passing through one of the three critical values $\rho_c(m)$ given in (\ref{eq:cusp-locations}). This happens at any of the three cusp-like singularities of $L_{N_0}$; there,  \textit{two} coinciding eigenvalues equal $\lambda_c(m)$. Each point $\rho_c(m)$ is a branch point of the double-valued function associated with the two coalescing eigenvalues \cite{Fik2020, Tanaka}, so it is impossible to distinguish between the two when reversing the process (i.e., when one follows the opposite trajectory of  $\rho=\epsilon_a$ in the complex plane). Finally, moving from the interior of $L_{N_0}$ to the exterior via a cusp-like singularity is associated with eigenvalue bifurcations \cite{Fik2020, Tanaka}. 

Further discussions on the physics of our example and connections to the interesting case $\alpha=0$,
can be found in \cite{Garmon}, which has a slightly different notation than ours.

\subsection{An equivalent matrix problem; possible generalizations}
The third-degree equation (\ref{eq:cubic-eq}) is the characteristic equation $\mathrm{det}[\lambda I-M(\rho)]=0$ of many $3\times 3$ matrices $M(\rho)$, an example being
\begin{equation}
\label{eq:matrixdefinition}
M(\rho)=\begin{bmatrix}
\rho+\alpha^2 & 0 & -\frac{\pi\alpha^2}{2} \\
1& \rho+\alpha^2 & 1\\
1 & \frac{\pi\alpha^2}{2} & 0 \\
\end{bmatrix}
\end{equation} 
The matrix $M(\rho)$ is an \textit{analytic matrix function} in the sense that all matrix elements are complex-analytic functions of $\rho$. Near any critical point
$\rho_c(m)$,  furthermore, the two coalescing eigenvalues are associated with a Puiseux series consisting of powers of $[\rho-\rho_c(m)]^{1/2}$ \cite{Fik2020}.  For the critical point $\rho_c(0)\in\mathbb{R}$, the Puiseux series is discussed in \cite{Tanaka} and \cite{Garmon}; importantly, the coefficient of the square-root term (i.e., the coefficient of $[\rho-\rho_c(0)]^{1/2}$) is nonzero.

For a general class of analytic matrix functions, eigenvalue behaviors \textit{near} such critical points are discussed in detail in Sections~3 and~4 of \cite{Fik2020}, with the Puiseux series being the main tool of study. While the conclusions are directly relevant to the example analyzed herein,\footnote{Corollary 3.5 of \cite{Fik2020}, for example, tells us that the cusp-like singularities in Fig. \ref{fig_6} are true cusps.} ref. \cite{Fik2020} mainly focuses on eigenvalue bifurcations. 

We believe that the behaviors exhibited by the $N$-level curves of $K_n(\rho)$ (and $M(\rho)$) are quite general. It might be possible to extend the Puiseux-series techniques in Sections~3 and~4 of \cite{Fik2020} so as to study the \textit{local} behaviors (near cusps, loops, and the like) of $N$-level curves of the aforementioned class of analytic matrix functions.  Preliminary work, however, indicates that this is not a simple task.

The type-2 eigenvalue contour lines of $K_8(\rho)$ (Fig. \ref{fig_5}) and the eigenvalue contour lines of $M(\rho)$  (Fig. \ref{fig_6}) are closed curves that share two important common features: (i) they are continuous; and (ii) they can be oriented in a natural manner. The $2\times 2$ Kac--Murdock--Szeg\H{o} matrix 
\begin{equation}K_2(\rho)=
\begin{bmatrix}
1& \rho  \\
\rho& 1 \\
\end{bmatrix}
\end{equation}
whose eigenvalues are $1-\rho$ and $1+\rho$, tells us that neither of these features is necessarily true in more general cases. The eigenvalue contour lines of $K_2(\rho)$ are two circles which intersect when $N>1$, but which are disjoint when $N<1$. Therefore, (a) the eigenvalue level curves (contour lines) $L_N$ are discontinuous whenever $N<1$, and (b) there seems to be no straightforward and natural method of orienting the \textit{composite} level curve consisting of both circles. Note that difficulty (b) arises because here, we are not distinguishing between type-1 and type-2 eigenvalues; the composite level curve in Fig.~\ref{fig:Fig_b5k} [consisting of both the solid (type-1) and the dashed (type-2) curves]  also presents difficulty (b).

For certain types of matrices (or polynomials) for which the aforementioned conditions (i) and (ii) are satisfied, it is logical to expect that equations such as (\ref{eq:winding-1}) and (\ref{eq:winding-2})---which involve the winding numbers of oriented level curves $L_N$ in the complex-$\rho$ plane---remain valid. We feel this matter is worthy of a systematic and rigorous study: It would be interesting to find matrix classes for which such equations remain true, and to state the equations in general form. To this end, a good starting point is the argument principle, see any text on functions of a single complex variable, or see  the introductory discussions  \cite{Baker} of degree theory in two dimensions. For any fixed $\rho\notin L_N$, this principle can be used to express what we have called $j_N(\rho)$ as a contour integral over the circle $|\lambda|=N$ of the logarithmic derivative of the characteristic polynomial.

\bibliographystyle{amsplain}

\end{document}